\documentclass{amsart}
\input{anc/preamble.sty}
\title{Fuss-Schr\"oder Paths and Rooted Plane Forests}
\author{Michael Kural}
\address{Massachusetts Institute of Technology}
\email{mkural@mit.edu}
\begin{document}
\begin{abstract}
We describe a bijection between $(k,k)$-Fuss-Schr\"oder paths of type $\lambda$ and certain rooted plane forests with $n(k+1)+2$ vertices. This yields a recursion which allows us to analytically enumerate the number of large $(k,r)$-Fuss-Schr\"oder paths of type $\lambda$, solving an open question posed by An, Jung, and Kim. Furthermore, we generalize the concept of $(k,r)$-Fuss-Schr\"oder paths to $(k,S)$-Fuss-Schr\"oder paths, in which $r$ can take any value in a given set $S$, and enumerate these paths as well.
\end{abstract}
\maketitle
\section{Introduction}
The $(k,r)$-Fuss-\Sch path was introduced by Eu and Fu (\cite{Eu08}) as a simultaneous generalization of Dyck paths, \Sch paths, and Fuss-Catalan paths. The paths were defined in their relation through enumerative invariants to generalized cluster complexes in a way aimed to specialize to the analogous relation between \Sch paths and simplicial associahedra at $k=1$ (\cite{Fomin05,Fomin03}). Eu and Fu enumerated the total number of small \krfs paths for a fixed $n$ and number of diagonal steps. Recently, An, Jung, and Kim (\cite{An17}) enumerated the number of small \krfs paths by fixed type and posed the same question for large \krfs paths of fixed type. Our central result is finding and proving such a formula.

\Sch paths provide one generalization for Dyck paths (north, east paths from $(0,0)$ to $(n,n)$ remaining nonstrictly above $y=x$) by also allowing diagonal steps. A path is \textit{small} if no diagonal step lies on the line $y=x$ and \textit{large} if no such restriction exists. Alternatively, Fuss-Catalan paths generalize Dyck paths by considering paths from $(0,0)$ to $(n,kn)$ nonstrictly above $y=kx$. In the way Fuss-Catalan paths are a Fuss analogue of Dyck (or Catalan) paths, \krfs paths are a Fuss analogue of \Sch paths in that they may use diagonal steps but also start at $(0,0)$ and end at $(n,kn)$, including both generalizations at once. The type of a given path denotes the partition yielded by consecutive runs of $E$ steps.

Enumeration of each of these paths have been studied in detail, both in total and by type. The number of Dyck paths by type was found by Krewaras (\cite{Kreweras72}) and proven with a direct bijective argument by Liaw (\cite{Liaw98}). An, Eu, and Kim (\cite{An14}) enumerated the number of large \Sch paths by type, and Park and Kim (\cite{Park16}) did the same for small \Sch paths by type. Finally, An, Jung, and Kim (\cite{An17}) enumerated the number of small \krfs paths by type.

In Section \ref{background} we describe the relevant definitions and past results in more detail. In Section \ref{rooted}, we define a certain class of rooted plane forests which are in bijection with \krfs paths. We discuss the recursive nature of these rooted plane forests before translating the relations to generating function identities in Section \ref{enumeration}. Using the Lagrange Inversion Formula, we extract the cardinality of several sets, including: small \krfs paths by type, large \krfs paths by type, and large \krfs paths ending in a diagonal step by type. Finally, in Section \ref{general}, we generalize our results to $(k,S)$-Fuss-\Sch paths, in which a diagonal step may occur at any $r \in S$ for some $S\subset[k]$.

We assume the reader is familiar with the theory of composition of generating functions-for reference, see \cite[Chapter 6]{Stanley99}.

\section{Background}
\begin{defn}
A \textit{Dyck path} of length $n$ is a path consisting of a sequence of east and north steps ($E=(1,0)$ and $N=(0,1)$, respectively) from $(0,0)$ to $(n,n)$ such that the path stays nonstrictly above the line $y = x$. 
\end{defn}
\begin{defn}
The \textit{type} of a Dyck path (or any path containing east steps) is the partition formed by parts with size the maximal consecutive runs of east steps in the path.
\end{defn}
Given a partition $\lambda$, we let $\ell(\lambda)$ denote its length and $|\lambda|$ denote the sum of its parts. We define $m_\lambda=m_{1}(\lambda)!m_{2}(\lambda)!m_{3}(\lambda)!\cdots$, where $m_i(\lambda)$ denotes the number of parts of $\lambda$ equal to $i$. Note that for a Dyck path of length $n$ and type $\lambda$, we have $|\lambda| = n$.
\begin{exmp}
The Dyck path shown in Figure \ref{dyck} has length $n=7$. Its type is $\lambda = (3,2,1,1)$, which satisfies $\ell(\lambda) = 4$, $|\lambda| =3+2+1+1= 7$, and $m_{\lambda} = 2!\cdot 1!\cdot 1! = 2$.
\end{exmp}
\begin{figure}[ht]
\centering
\includegraphics[width=0.3\textwidth]{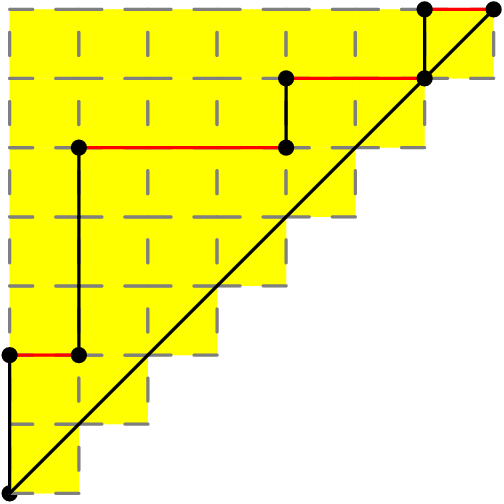}
\caption{The Dyck path $NNENNNEEENEENE$ of length $7$.}
\label{dyck}
\end{figure}
It is well known that the number of Dyck paths of a given length $n$ is $\displaystyle C_n = \frac{1}{n+1}\binom{2n}{n}$. However, one can also restrict to paths of a fixed type:
\begin{thm}[Krewaras, \cite{Kreweras72}]
The number of Dyck paths of type $\lambda$ is 
\[
\frac{n(n-1)\cdots(n-\ell(\lambda)+2)}{m_\lambda}.
\]
\end{thm}

A slightly more general type of path allows diagonals.
\begin{defn}
A \textit{large \Sch path} of length $n$ is a path with north, east, and diagonal steps (where a diagonal step is $(1,1)$) from $(0,0)$ to $(n,n)$ which stays nonstrictly above the line $y=x$.

A \textit{small \Sch path} of length $n$ is a large \Sch path of length $n$ such that no diagonal step lies on $y=x$.
\end{defn}

The number of large and small \Sch paths by type has much more recently been enumerated.

\begin{thm}[An, Eu, Kim, \cite{An14}]
The number of large \Sch paths of type $\lambda$ is 
\[
\frac{1}{|\lambda|+1}\binom{n}{|\lambda|}\binom{n+1}{\ell(\lambda)}\frac{\ell(\lambda)!}{m_\lambda}.
\]
\end{thm}
\begin{thm}[Park, Kim, \cite{Park16}]
The number of small \Sch paths of type $\lambda$ is
\[
\frac{1}{n+1}\binom{n-1}{|\lambda|-1}\binom{n+1}{\ell(\lambda)}\frac{\ell(\lambda)!}{m_\lambda}.
\]
\end{thm}

We may also generalize Dyck paths in an entirely different sense. Instead of altering the types of steps, we alter the region that contains the path.

\begin{defn}
A \textit{$k$-Fuss-Catalan path} of length $n$ is a path with north and east steps from $(0,0)$ to $(n,kn)$ which stays nonstrictly above the line $y = kx$.
\end{defn}
The number of Fuss-Catalan paths has also been counted by type.
\begin{thm}[Armstrong, \cite{Armstrong09}]
The number of $k$-Fuss-Catalan paths of type $\lambda$ is
\[
\frac{(kn)!}{m_\lambda\cdot (kn+1-\ell(\lambda))!}.
\]
\end{thm}

Finally, we come to our main object of study, the \krfs path. This generalizes all three paths thus far described.
\begin{defn}
For integers $k$ and $r$ such that $1\le r \le k$, a \textit{large \krfs path} of length $n$ is a path from $(0,0)$ to $(n,kn)$ using east, north, and diagonal steps which stays nonstrictly above the line $y=kx$ and such that each diagonal step can only go from the line $y=kj+r-1$ to the line $y= kj+r$ for some $j$.

A \textit{small \krfs path} of length $n$ is a large \krfs path of length $n$ such that no diagonal steps touch the line $y=kx$.

As before, the type of a \krfs path is the partition formed by maximal consecutive runs of east steps.
\end{defn}
In the way $k$-Fuss-Catalan paths are a Fuss analogue for Dyck (or Catalan) paths, \krfs paths are a Fuss analogue for \Sch paths.

An, Jung, and Kim (\cite{An17}) enumerated the number of small \krfs paths of type $\lambda$, which is independent of $r$.
\begin{thm}
The number of small \krfs paths of type $\lambda$ is
\[
\frac{1}{kn+1}\binom{n-1}{|\lambda|-1}\binom{kn+1}{\ell(\lambda)}\frac{\ell(\lambda)!}{m_\lambda}.
\]
\end{thm}
They posed the open question of enumerating large \krfs paths of type $\lambda$, which is our main focus. Note that for $1\le r \le k-1$, all paths are small, so in these cases the number of large \krfs paths of type $\lambda$ has already been enumerated. Thus we only consider the case $r=k$.
\label{background}
\section{A Rooted Plane Forest Bijection}
\label{rooted}
An, Jung, and Kim (\cite{An17}) provide bijections between various \Sch and Dyck paths and noncrossing set partitions and conjecture a similar bijection for large \kkfs paths.

We approach the enumeration problem from the point of view of trees, which we view as a natural way to interpret the data in each path. In particular, we establish a bijection between \kkfs paths of type $\lambda$ and rooted plane forests (or ordered forests). The bijective algorithm is analogous to the algorithm given in \cite{An17} relating \kkfs paths and sparse noncrossing partitions.

\begin{defn}
A \textit{rooted plane tree} (or \textit{ordered tree}) is a rooted tree such that each vertex has a linearly ordered list of children. It can be defined recursively as a root with an ordered list of rooted plane subtrees which are rooted at each child of the root.
\end{defn}
\begin{defn}
A \textit{rooted plane forest} (or \textit{ordered forest}) is a linearly ordered list of rooted plane trees.
\end{defn}
Note that both rooted plane trees and rooted plane forests have natural \textit{pre-order} labelings stemming from the pre-order transversal. (Recall that the labeling algorithm $\pre(r)$ on a root $r$ of a rooted plane tree can be defined by the following steps:
\begin{enumerate}
\item Give $r$ the label $\ell+1$ if $\ell$ is the last label which has been given, or otherwise the label $0$ if no label has been given yet.
\item If $(c_1, c_2,\ldots,c_k)$ is the ordered list of children of $r$, then for $1\le i \le k$, call $\pre(c_i)$.
\end{enumerate}
For a rooted plane forest with an ordered list $(r_1, r_2,\ldots,r_k)$ of plane tree roots, call $\pre(r_1), \pre(r_2), \ldots, \pre(r_k)$ in order.) Our labeling convention is $0$-indexed so that the $V$ vertices are labeled with $0,1,2,\ldots, V-2$, and $ V-1$. 
\begin{defn}
\label{nkforest}
An $(n,k)$-\textit{valid forest} is a rooted plane forest with $2+(k+1)n$ vertices and $(k+1)n$ edges satisfying the following properties:
\begin{enumerate}
\item The forest is the union of two rooted plane trees.
\item If a vertex $v$ has label $s$ with $s\equiv 1\pmod{k+1}$, then $v$ has either $0$ or $k+1$ children.
\item If a vertex $v$ has label $s$ with $s\not \equiv 1\pmod{k+1}$, then $v$ has $m(k+1)$ children for some nonnegative integer $m$.
\end{enumerate}
We say the \textit{type} of an \nk forest is the partition formed by all parts $m$ (with repetition) corresponding to some $v$ in the latter category with $m(k+1)$ children.
\end{defn}
\begin{exmp}
An example of an \nk forest with pre-order labeling is given in Figure \ref{fig:nkforest}. In this case $n=4$ and $k=2$, and so our distinguished vertices are those labeled $1\pmod{3}$. Since the vertex labeled $0$ has $3$ children and the vertex labeled $5$ has $6$ children, the type of the forest is $\lambda = (2,1)$.
\end{exmp}
\begin{figure}[ht]
\centering
\includegraphics[width=0.5\textwidth]{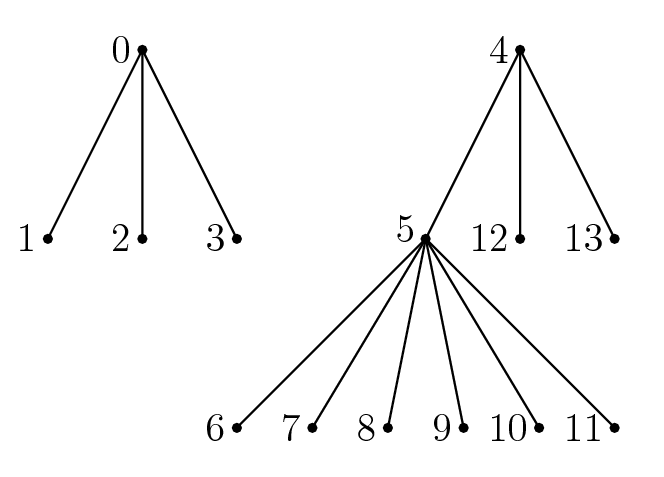}
\caption{A $(4,2)$-valid forest.}
\label{fig:nkforest}
\end{figure}

\begin{lemma}
\label{forestbijection}
There is a type preserving bijection between large \kkfs paths and \nk forests. Furthermore, small paths correspond exactly to \nk forests in which the second tree is a single vertex, and paths ending with a diagonal correspond to \nk forests in which the first tree is a single vertex.
\end{lemma}
\begin{proof}

As noted in \cite{An17}, we may first biject a \kkfs path to a sequence giving the heights of $D$ and $E$ steps. In the rectangle from $(0,0)$ to $(n,kn)$ in which the path is contained, we label the $(n-i)k$th row by $i(k+1)+1$ for $0\le i\le n-1$ and the horizontal line $y= (n-i)k+j$ by $i(k+1)-j$ for $0\le j <k$. (Note that we differ from the notation of \cite{An17} in that we use indices beginning with $0$.)

Reading the labels of $D$s and $E$s in increasing order (equivalently, from right to left) yields a nondecreasing sequence $0\le s_1\le s_2\le \cdots \le s_n$ which completely characterizes the path. Note that $s_i\le (i-1)(k+1)+1$ for all $i$ and that if $s_i\equiv 1\pmod{k+1}$ for some $i$, then the value of $s_i$ is not repeated. These correspond to the conditions that the path stays above the diagonal $y=kx$ and that $D$ steps can only occur between the lines $y=kj+r-1$ and $y=kj+r$, respectively. Conversely, it is not hard to see that from any nondecreasing integer sequence satisfying these two conditions, we may reconstruct its corresponding \kkfs path.

We can read out the type of the corresponding path by considering sequences of $m$ consecutive labels $j$ such that $j\not\equiv 1\pmod{k+1}$: each such $m$ is a part of the partition.

Now we describe a bijection between such sequences $s_i$ and \nk forests. Given a sequence $s_1\le s_2\le \cdots \le s_n$, consider the following algorithm:
\begin{enumerate}
\item Begin with two independent vertices, labeled $0$ and $1$.
\item Successively read elements of the sequence. If the following $m$ terms of the sequence are equal to $j$,
\begin{enumerate}
\item add $m(k+1)$ children to the vertex labeled $j$;
\item maintain the pre-order labeling by incrementing the label of each  $i>j$ by $m(k+1)$ and labeling the children of $j$ with $j+1,j+2,\ldots, j+m(k+1)$.
\end{enumerate}
\end{enumerate}

Note that at each step of the process, a multiple of $k+1$ children are added to a single vertex, which implies that the pre-order labelings $\pmod{k+1}$ are preserved. This implies that the property of a vertex being $1\pmod{k+1}$ or not is preserved throughout the process, and so there are no issues regarding which vertices can gain greater than $(k+1)$ children. Additionally, if a vertex $j$ adds a number of children, the vertices $0,1,2,\ldots ,j$ are never altered again in labels or children. Thus at the end of the process, the labels on vertices with positive numbers of children correspond exactly to the sequence $s_i$, accounting for multiplicity given $m(k+1)$ children.

It then only suffices to show that if 
\[
s_{j-1}<s_j=s_{j+1}=\cdots =s_{j+m-1}\le (j-1)(k+1)+1,
\]
then the vertices from $s_{j-1}+1$ to $(j-1)(k+1)+1$ already exist when the children of $s_j$ are created, and conversely if the sequence $s_j$ is read out in pre-order from the vertices with children in the \nk forest, then $s_j \le (j-1)(k+1)+1$. The first assertion holds because there are $(j-1)(k+1)+2$ vertices in the forest immediately before $s_j$ is read, and in particular the vertices from $s_{j-1}+1$ to $(j-1)(k+1)+1$ are available to be chosen. For the same reason, in the reverse direction the maximum integer $s_j$ could be recorded as is $(j-1)(k+1)+1$.

Finally, it is clear that the type of a \kkfs path corresponds to the type of an \nk forest.

\end{proof}
\begin{figure}[ht]
\centering

\begin{exmp}
An example of the process illustrating the bijection between an \nk forest and a \kkfs path corresponding to the same sequence is carried out in Figure \ref{kkfs-biject} and Figure \ref{nkforest-biject}. In this case, $n=4, k=2$, $\lambda = (2,1)$, and $(s_1,s_2,s_3,s_4) = (0,4,5,5)$. Note that $s_i \le (i-1)(k+1) + 1$ for all $i$, and equality holds for $i=2$, which corresponds to the fact that the $D$ step touches the diagonal $y=2x$.
\end{exmp}

\begin{subfigure}[t]{0.23\textwidth}
\centering
\includegraphics[height=2in]{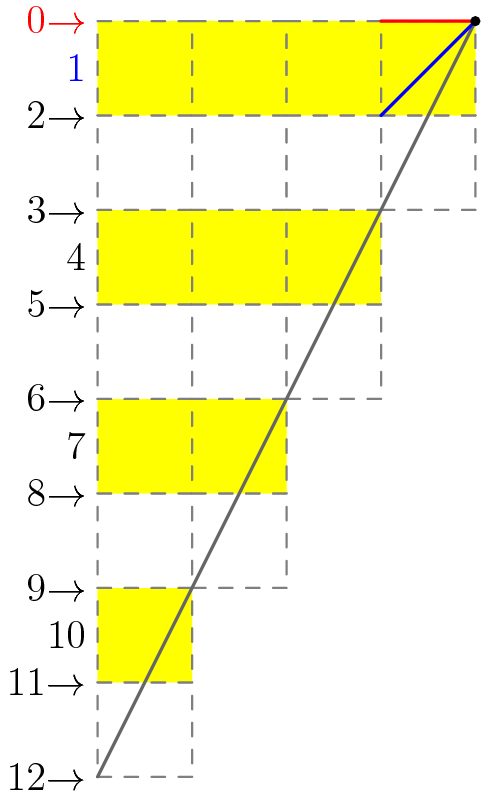}
\end{subfigure}
~
\begin{subfigure}[t]{0.23\textwidth}
\centering
\includegraphics[height = 2in]{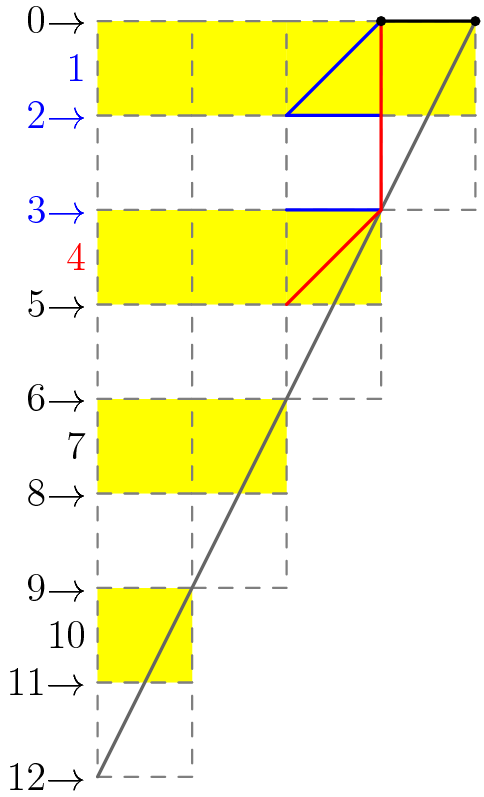}
\end{subfigure}
~
\begin{subfigure}[t]{0.23\textwidth}
\centering
\includegraphics[height =2in]{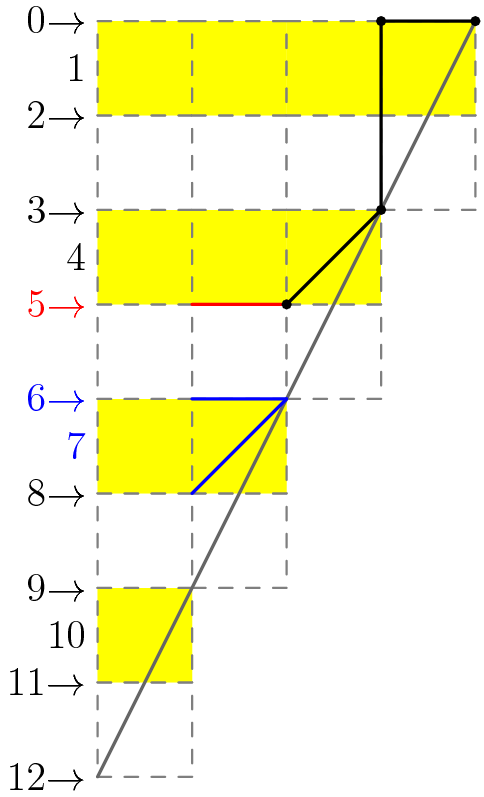}
\end{subfigure}
~
\begin{subfigure}[t]{0.23\textwidth}
\centering
\includegraphics[height = 2in]{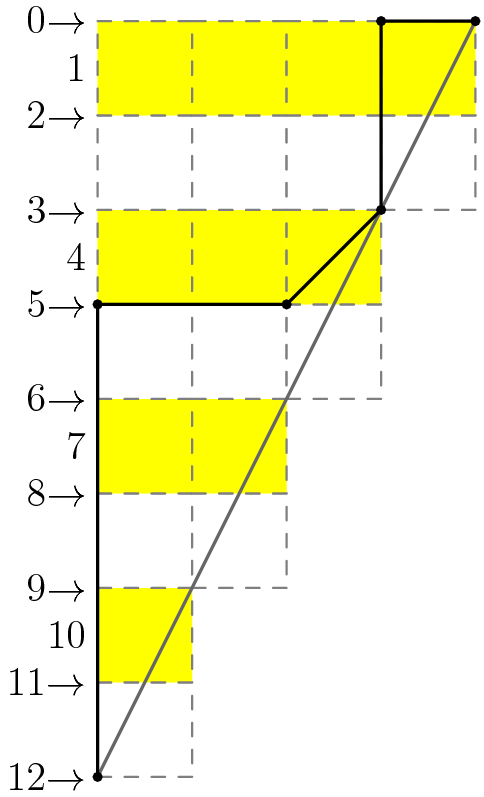}
\end{subfigure}

\caption{The $(2,2)$-Fuss-\Sch path of length $4$ corresponding to the sequence $(s_1, s_2, s_3, s_4) = (0,4,5,5)$.}
\label{kkfs-biject}
\end{figure}
\begin{figure}[ht]
\centering

\begin{subfigure}[t]{0.4\textwidth}
\centering
\includegraphics[width = \textwidth]{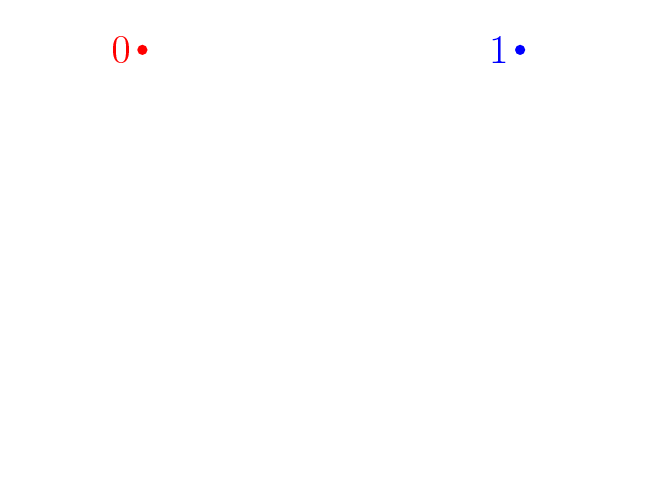}
\end{subfigure}
\begin{subfigure}[t]{0.4\textwidth}
\centering
\includegraphics[width = \textwidth]{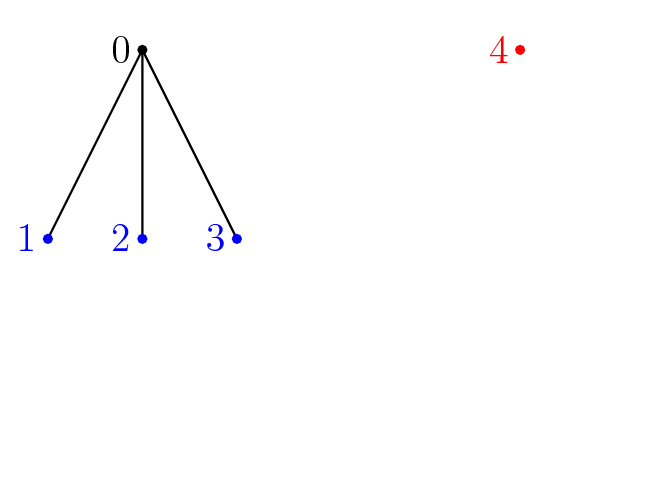}
\end{subfigure}

\begin{subfigure}[t]{0.4\textwidth}
\centering
\includegraphics[width = \textwidth]{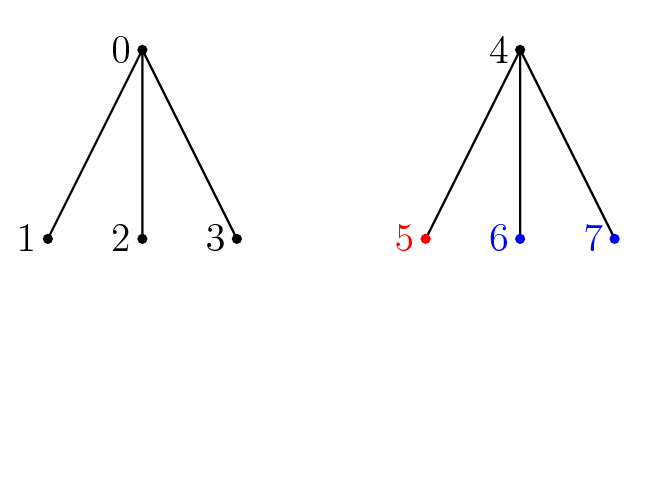}
\end{subfigure}
\begin{subfigure}[t]{0.4\textwidth}
\centering
\includegraphics[width = \textwidth]{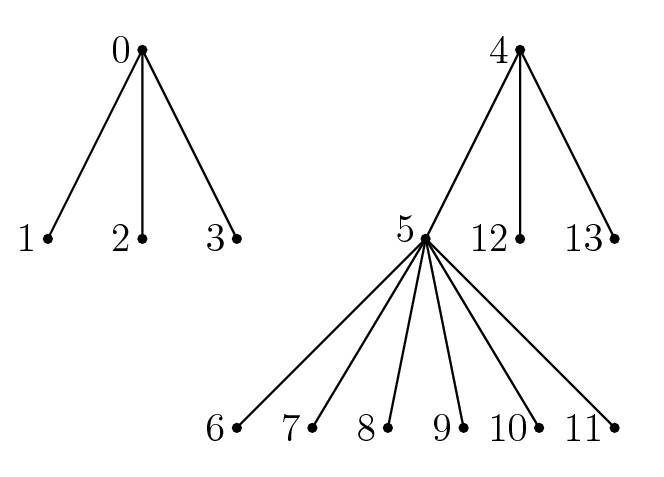}
\end{subfigure}
\caption{The $(4,2)$-valid forest corresponding to the sequence $(s_1,s_2,s_3,s_4) = (0,4,5,5)$.}
\label{nkforest-biject}
\end{figure}

Now to enumerate the number of \nk forests, note that the only significant aspect of the labeling is the reduction of the labels $\pmod{k+1}$. We therefore count the number of trees with root labeled in a specific residue class $\pmod{k+1}$, which is a well-defined notion.
\begin{defn}
A \textit{$j$-rooted \nk tree} is a rooted plane tree which is labeled in pre-order, except in that the labeling begins with $j$ at the root and is only considered $\pmod{k+1}$. Furthermore, the latter two conditions of Definition \ref{nkforest} hold:
\begin{enumerate}
\item If a vertex $v$ has label $s$ with $s \equiv 1 \pmod{k+1}$, then $v$ has either $0$ or $k+1$ children.
\item If a vertex $v$ has label $s$ with $s\not\equiv 1\pmod{k+1}$, then $v$ has $m(k+1)$ children for some nonnegative integer $m$.
\end{enumerate}
The \textit{type} of a $j$-rooted \nk tree is the partition formed by all parts $m$ such that some $v$ in the latter category has $m(k+1)$ children.
\end{defn}

An \nk forest consists of a $0$-rooted $(n_1,k)$-valid tree and a $1$-rooted $(n_2,k)$-valid tree for nonnegative integers $n_1,n_2$ such that $n_1+n_2=n$. The union of the types (as multisets) of the trees yields the type of the forest.

Let $S_j$ denote the set of $j$-rooted \nk trees. We then note that there is a recursive definition for elements of $S_j$. If $j\not \equiv 1\pmod{k+1}$, then for some nonnegative integer $m$, an element $T \in S_j$ consists of a root labeled $j$ and $m(k+1)$ children with subtrees in 
\[
\underbrace{S_{j+1},S_{j+2},\ldots ,S_{k},S_{0},S_{1},\ldots, S_{j}}_{m\text{ times}},
\]
respectively. If $j\equiv 1 \pmod{k+1}$, then an element $T \in S_1$ consists of a root labeled $1$ and either $0$ or $(k+1)$ children with subtrees in 
\[
S_2,S_3,\ldots , S_{k-1},S_{k},S_0,S_1,
\]
respectively. This relation will allow us to enumerate elements of $S_j$ by type recursively.

\section{Enumeration of Trees}
\label{enumeration}

To solve the recursion analytically, we appeal to the classical Lagrange Inversion Formula (see, for example, \cite[p. 42]{Stanley99}), which allows us to find the coefficients of the compositional inverse of a generating function.
\begin{thm}[Lagrange Inversion Formula]
Let $K$ be a field of characteristic $0$, and suppose $f(x),G(x) \in K[[x]]$ such that $G(0)\neq 0$ and $f(x) = xG(f(x))$. Then for any power series $H(x) \in K[[x]]$ and any positive integer $n$,
\[
[x^{n}] H(f(x)) = \frac{1}{n} [x^{n-1}] H'(x) G(x)^{n}.
\]
where the operator $[x^n]$ gives the coefficient of the $x^n$ term of a power series.
\end{thm}

We apply this powerful tool to enumerate the number of Fuss-Schr\"oder paths of given type.
Define the field $K=\C(t_1,t_2,t_3,\ldots)$ given by rational functions of finitely many $t_i$ over $\C$. We write $\T^{\lambda}\in K$ to denote $\T^{\lambda}=t_1^{m_1(\lambda)}t_2^{m_2(\lambda)}t_3^{m_3(\lambda)}\cdots$, where again $m_i(\lambda)$ denotes the number of parts of $\lambda$ of size $i$. We also define the function $\Theta(x)$ and prove some of its properties, which will be useful in our analytic approach.
\begin{defn}
Let $\Theta(x) \in K[[x]]$ be defined as
\[
\Theta(x) = t_1x+t_2x^2+t_3x^3+\cdots .
\]
\end{defn}
\begin{lemma}
\label{thetapower}
We have
\[
\Theta(x)^{m} =\sum_{n\ge 1} \sum_{\substack{|\lambda|=n\\ \ell(\lambda)=m}}\frac{m!}{m_\lambda}   \T^{\lambda} x^{n}
\]
\end{lemma}
\begin{proof}
Each term in $\Theta(x)^{m}$ is the product of an ordered $m$-tuple of terms in the form $t_i x^{i}$. But each term is in the form $\T^{\lambda} x^{n}$ for some partition $\lambda$ induced by the product of the $t_i$ and for some $n$. Note $|\lambda|=n$ and $\ell(\lambda)=m$. Conversely, for each partition $\lambda$, there is an ordered $m$-tuple corresponding to $\lambda$ for each ordered arrangement of the $m_1+m_2+m_3+\cdots = m$ parts, where each of the $m_i$ parts of size $i$ are indistinguishable. It isn't hard to see that there are then $\displaystyle\binom{m}{m_1,m_2,\ldots, m_\lambda}= \frac{m!}{m_\lambda}$ $m$-tuples corresponding to each $\lambda$, and so the identity holds.
\end{proof}
As a direct consequence, we get the following, which will aid in our computation.
\begin{lemma}
\label{technical}
Let $a$ and $b$ be nonnegative integers and let $n\ge 1$. Then
\[
[x^{n-1}](1+\Theta(x))^{a}(1+x)^{b} = \sum_{\lambda} \binom{b}{n-1-|\lambda|} \binom{a}{\ell(\lambda)}\frac{\ell(\lambda)!}{m_\lambda} \T^{\lambda}
\]
and 
\[
[x^{n-1}] \Theta'(x)(1+\Theta(x))^{a}(1+x)^{b} = \sum_{\lambda} \frac{|\lambda|}{a+1} \binom{b}{n-|\lambda|}\binom{a+1}{\ell(\lambda)}\frac{\ell(\lambda)!}{m_\lambda}\T^{\lambda}.
\]
\end{lemma}
\begin{proof}
By Lemma \ref{thetapower}, we have
\[
[x^{n-1}] (1+\Theta(x))^{a}(1+x)^{b} = [x^{n-1}](1+x)^{b}\sum_{\lambda}  \binom{a}{\ell(\lambda)}\frac{\ell(\lambda)!}{m_\lambda} \T^{\lambda} x^{|\lambda|}
\]
which implies the first equation.

Similarly, we have
\begin{align*}
[x^{n-1}]\Theta'(x)(1+\Theta(x))^{a}(1+x)^{b}
&=\frac{1}{a+1} (1+x)^{b} \frac{d}{dx} (1+\Theta(x))^{a+1}\\
&= [x^{n-1}] \frac{1}{a+1} (1+x)^{b} \sum_{\lambda} \binom{a+1}{\ell(\lambda)}\frac{\ell(\lambda)!}{m_\lambda}\T^{\lambda} |\lambda| x^{|\lambda|-1}\\
&= \sum_{\lambda}\frac{|\lambda|}{a+1} \binom{b}{n-|\lambda|} \binom{a+1}{\ell(\lambda)}\frac{\ell(\lambda)!}{m_\lambda} \T^{\lambda}
\end{align*}
as desired.
\end{proof}

Having established these computational lemmas, we now return to our enumeration problem.
\begin{defn}
Let $j\le k$ be a nonnegative integer. We define 
\[
A_j(x) = \sum_{T \in S_j} \T^{\lambda}x^{\nu(T)} 
\]
where $\nu(T)$ denotes the $n$ such that $T$ has $(k+1)n$ edges and $(k+1)n+1$ vertices.
\end{defn}
\begin{exmp}
We have 
\[
A_0(x) = 1+t_1x+(t_1+kt_1^2+t_2)x^2+\cdots
\]
and
\[
A_1(x) = 1+x+(1+kt_1)x^2+\cdots .
\]
\end{exmp}
The recursive structure of the elements of each $S_j$ yields the following recursion for $A_j(x)$:
\begin{lemma}
If $j \not \equiv 1\pmod{k+1}$, then
\[
A_j(x) = 1+t_1C(x)+t_2C(x)^2+t_3C(x)^3+\cdots
\]
where $C(x) = xA_0(x)A_1(x)\cdots A_{k}(x)$.

If $j\equiv 1 \pmod{k+1}$, then
\[
A_1(x) = 1 + C(x).
\]

Furthermore, the $\T^{\lambda} x^{n}$ coefficients of $A_0(x),A_1(x),$ and $A_0(x)A_1(x)$ count the number of small \kkfs paths, large \kkfs paths ending in a diagonal, and large \kkfs paths, respectively, each of type $\lambda$.
\end{lemma}
\begin{proof}
All statements are direct corollaries of the recursive properties of elements of $S_j$ and Lemma \ref{forestbijection} along with the theory of composition of formal power series.
\end{proof}

Now the first equality implies $A_0(x)=A_2(x)=A_3(x)=\cdots = A_{k}(x)$, so we denote this common power series by $A(x)$. We also define $B(x)=A_1(x)$. Then the lemma can be restated in the following way.
\begin{lemma}
\label{relations}
We have
\[
A_{j}(x) = \begin{cases} A(x) &j\not \equiv 1 \pmod{k+1}  \\
B(x) &j \equiv 1 \pmod{k+1} \end{cases}
\]
where $A(x),B(x),$ and $C(x) \in K[[x]]$ satisfy the relations
\begin{align*}
A(x)&=1+\Theta(C(x))\\
B(x)&=1+C(x)\\
C(x)&=x A(x)^{k}B(x).
\end{align*}
\end{lemma}

Given these relations, we can use the Lagrange Inversion Formula to solve for the coefficients of $A(x)$, $B(x)$, and $A(x)B(x)$.
\begin{thm}
We have
\begin{align*}
A(x)&= 1+\sum_{n\ge 1}\sum_{\lambda}\frac{1}{n} \frac{1}{kn+1}\binom{n}{|\lambda|}\binom{kn+1}{\ell(\lambda)}\frac{\ell(\lambda)!}{m_\lambda}\T^{\lambda} x^n\\
B(x) &=1+\sum_{n\ge 1}\sum_{\lambda} \frac{1}{n} \binom{n}{|\lambda|+1}\binom{kn}{\ell(\lambda)}\frac{\ell(\lambda)!}{m_\lambda}\T^{\lambda} x^n\\
A(x)B(x)&=1+\sum_{n \ge 1}\sum_{\lambda} \frac{1}{n}\left(\binom{n}{|\lambda|+1}+\frac{|\lambda|}{kn+1}\binom{n+1}{|\lambda|+1}\right)\binom{kn+1}{\ell(\lambda)}\frac{\ell(\lambda)!}{m_\lambda}\T^{\lambda} x^n.
\end{align*}
\end{thm}
\begin{proof}
We use the Lagrange Inversion Formula along with Lemma \ref{relations} and Lemma \ref{technical}. Because
\[
C(x) = xA(x)^{k}B(x) =x (1+\Theta(C(x))^{k}(1+C(x)),
\]
we'll use the Lagrange Inversion Formula with $f(x) = C(x)$, $G(x) =x(1+\Theta(x))^{k}(1+x)$, and varying values of $H(x)$.

Note that $A(x) = H(C(x))$ for $H(x) = 1+\Theta(x)$, so for $n\ge 1$,
\begin{align*}
[x^{n}]A(x) &=\frac{1}{n} [x^{n-1}]\Theta'(x)(1+\Theta(x))^{kn}(1+x)^{n}\\
&=\frac{1}{n}\sum_{\lambda} \frac{1}{kn+1}\binom{n}{|\lambda|}\binom{kn+1}{\ell(\lambda)}\frac{\ell(\lambda)!}{m_\lambda}\T^{\lambda}.
\end{align*}
Now $B(x) = H(C(x))$ for $H(x)=1+x$, so for $n \ge 1$,
\begin{align*}
[x^{n}]B(x) = \frac{1}{n}\sum_{\lambda} \binom{n}{|\lambda|+1} \binom{kn}{\ell(\lambda)}\frac{\ell(\lambda)!}{m_\lambda}\T^{\lambda}.
\end{align*}
Finally, $A(x)B(x) = H(C(x))$ for $H(x) =(1+\Theta(x))(1+x)$, which has derivative $H'(x) = \Theta'(x)(1+x)+(1+\Theta(x))$. Thus
\begin{align*}
[x^{n}]A(x)B(x) &= \frac{1}{n}[x^{n-1}]\Theta'(x)(1+\Theta(x))^{kn}(1+x)^{n+1}+(1+\Theta(x))^{kn+1}(1+x)^{n}\\
&=\sum_{\lambda} \frac{1}{n}\left(\binom{n}{|\lambda|+1}+\frac{|\lambda|}{kn+1}\binom{n+1}{|\lambda|+1}\right)\binom{kn+1}{\ell(\lambda)}\frac{\ell(\lambda)!}{m_\lambda}\T^{\lambda} .
\end{align*}
\end{proof}

In particular, we recover the formula for the number of small \kkfs paths of type $\lambda$ and find two new formulae.
\begin{cor}
For $n \ge 1$, the number of large \kkfs paths of type $\lambda$ is 
\[
\frac{1}{n}\left(\binom{n}{|\lambda|+1}+\frac{|\lambda|}{kn+1}\binom{n+1}{|\lambda|+1}\right)\binom{kn+1}{\ell(\lambda)}\frac{\ell(\lambda)!}{m_\lambda}.
\]

The number of large \kkfs paths of type $\lambda$ which end in a diagonal step is
\[
\frac{1}{n} \binom{n}{|\lambda|+1}\binom{kn}{\ell(\lambda)}\frac{\ell(\lambda)!}{m_\lambda}.
\]
\end{cor}
\section{Generalizing to Subsets of \texorpdfstring{$[k]$}{}}
\label{general}
In the original definition of \krfs paths, a diagonal step is only allowed to go from the line $y = kj+r-1$ to the line $y=kj+r$ for some $j$. In this section, we generalize the definition and enumeration results to paths for which diagonal steps are allowed to occur at a specific subset of residues $\pmod{k}$.
\begin{defn}
For a subset $S \subseteq [k]$, a \textit{large $(k,S)$-Fuss-\Sch path} of length $n$ is defined to be a path from $(0,0)$ to $(n,kn)$ using east steps, north steps, and diagonal steps such that the path never crosses below the line $y =kx$ and the diagonal steps are only allowed to go from the line $y = kj+r-1$ to the line $y=kj+r$ for some $j$ if $r \in S$. 

A \textit{small $(k,S)$-Fuss-\Sch path} is a large $(k,S)$-Fuss-\Sch path such that no diagonal step touches the main diagonal $y=kx$. 

The \textit{type} of such a path is the partition with parts given by maximal consecutive runs of east steps.
\end{defn}

We may prove a sequence of analogous results for $(k,S)$-Fuss-\Sch paths. Because of the similarity, some details are omitted.
\begin{defn}
Given a subset $S\subseteq [k]$ such that $|S|=d\le k$, we define a subset $S'\subseteq [k+d]$ as follows:

Write the numbers $0,1,2,3,\ldots,k-1$ in order, and for each $r \in S$, place a marker after $k-r$ in the list. The subset $S'$ is given by the set of indices of markers in the length $k+d$ list (with first index $0$).
\end{defn}
The subset $S'$ corresponds to the subset of residue classes the sequence $s_i$ can take $\pmod{k+d}$ which correspond to diagonal steps.
\begin{defn}
For positive integers $n,k$ and a subset $S \subseteq [k]$ with $d=|S|$, a \textit{\nks forest} is a rooted plane forest with pre-order labeling defined as follows:
\begin{enumerate}
\item If a vertex $v$ has label $s$ with $s\equiv r \pmod{k+d}$ for some $r \in S'$, then $v$ either has $0$ or $k+d$ children.
\item If a vertex $v$ has label $s$ such that $s\not\equiv r$ for all $r \in S'$, then $v$ has $m(k+d)$ children for some nonnegative integer $m$.
\item 
\begin{itemize}
\item If $S$ contains $k$, then the forest has $n(k+d)+2$ vertices, $n(k+d)$ edges, and is the union of two rooted plane trees.
\item If $S$ does not contain $k$, then the forest is a rooted plane tree with $n(k+d)+1$ vertices and $n(k+d)$ edges.
\end{itemize}
\end{enumerate}
The type of an \nks forest is defined to be the partition formed by all parts $m$ (with repetition) such that some $v$ in the latter category has $m(k+1)$ children.

\end{defn}
\begin{lemma}
There is a type preserving bijection between large $(k,S)$-Fuss-\Sch paths and \nks forests. Small paths correspond to \nks forests in which the second tree is a single vertex, and paths ending with a diagonal step correspond to \nks forests in which the first tree is a single vertex. (In particular, if $k \not \in S$, the corresponding path is always small and never ends with a diagonal step.)
\end{lemma}
\begin{proof}
The proof is almost the same as that of Lemma \ref{forestbijection}. Note that if $k \not \in S$, the corresponding sequence $s_i$ satisfies $s_i \le (i-1)(k+d)$, while if $k \in S$, $s_i$ satisfies $s_i \le (i-1)(k+d)+1$. This accounts for the difference in number of beginning vertices.
\end{proof}
\begin{defn}
A \textit{$j$-rooted \nks tree} is a rooted plane tree labeled in pre-order $\pmod{k+d}$ and starting with $j$ at the root, such that 
\begin{enumerate}
\item If a vertex $v$ has label $s$ with $s \equiv r \pmod{k+d}$ for some $r \in S'$, then $v$ has either $0$ or $k+d$ children.
\item If a vertex $v$ has label $s$ such that $s \not\equiv r$ for all $r \in S'$, then $v$ has $m(k+d)$ children for some nonnegative integer $m$.
\end{enumerate}
The type of a $j$-rooted \nks tree is as in previous definitions.
\end{defn}

Now we see that an \nks forest is the union of a $0$-rooted $(n_1,k,S)$-valid tree and a $1$-rooted $(n_2,k,S)$-valid tree for some $n_1+n_2=n$ when $k \in S$, and when $k \not \in S$, it is simply a $0$ rooted \nks tree.

\begin{defn}
Let
\[
A_S(x) = \sum_{T}\T^{\lambda} x^{\nu(T)}
\]
where the sum is over all $j$-rooted \nks trees $T$ for some fixed $j\not \in S'$. Let
\[
B_S(x) = \sum_{T} \T^{\lambda} x^{\nu(T)}
\]
where the sum is over all $j$-rooted \nks trees $T$ for some fixed $j \in S'$. (If there is no $j \in S'$, then $B_S(x)=1$. As before, all elements of $S'$ and all elements of $[k+d]\backslash S'$ yield equal generating functions.)

Let 
\[
C_S(x) = x A_S(x)^{k}B_S(x)^{d}.
\]
\end{defn}

Now by our recursion,
\[
A_S(x) = 1+\Theta(C_S(x))
\]
and
\[
B_S(x) =1+C_S(x).
\]
\begin{thm}
We have
\begin{align*}
A_S(x) &= 1+\sum_{n\ge 1} \frac{1}{n}\sum_{\lambda}\frac{|\lambda|}{kn+1}\binom{dn}{n-|\lambda|}\binom{kn+1}{\ell(\lambda)}\frac{\ell(\lambda)!}{m_\lambda} \T^{\lambda} x^n\\
B_S(x)&= 1+\sum_{n\ge 1} \frac{1}{n}\sum_{\lambda} \binom{dn}{n-1-|\lambda|}\binom{kn}{\ell(\lambda)}\frac{\ell(\lambda)!}{m_\lambda}\T^{\lambda}x^{n}\\
A_S(x)B_S(x) &= 1 + \sum_{n \ge 1} \frac{1}{n}\sum_{\lambda} \left(\binom{dn}{n-1-|\lambda|}+\frac{|\lambda|}{kn+1}\binom{dn+1}{n-|\lambda|}\right)\binom{kn+1}{\ell(\lambda)}\frac{\ell(\lambda)!}{m_\lambda}\T^{\lambda}x^{n}.
\end{align*}
\end{thm}
\begin{proof}
Apply the Lagrange Inversion Formula and Lemma \ref{technical} in the same way as before.
\end{proof}
\begin{cor}
For a set $S$ not containing $k$, the number of large $(k,S)$-Fuss-\Sch paths of type $\lambda$ is
\[
\frac{1}{n}\frac{|\lambda|}{kn+1}\binom{dn}{n-|\lambda|}\binom{kn+1}{\ell(\lambda)}\frac{\ell(\lambda)!}{m_\lambda}.
\]
For a set $S$ containing $k$, the number of small $(k,S)$-Fuss-\Sch paths of type $\lambda$ is the same quantity. The number of large paths ending in a diagonal step is
\[
\frac{1}{n}\binom{dn}{n-1-|\lambda|}\binom{kn}{\ell(\lambda)}\frac{\ell(\lambda)!}{m_\lambda}.
\]
The total number of large $(k,S)$-Fuss-\Sch paths of type $\lambda$ for a set $S$ containing $k$ is
\[
\frac{1}{n} \left(\binom{dn}{n-1-|\lambda|}+\frac{|\lambda|}{kn+1}\binom{dn+1}{n-|\lambda|}\right)\binom{kn+1}{\ell(\lambda)}\frac{\ell(\lambda)!}{m_\lambda}.
\]
\end{cor}
Interestingly, these expressions depend only on the cardinality of $S$.
\section{Conclusion}
In recent years, much work has been done on enumerating paths of a fixed type, including Dyck paths, \Sch paths, and Fuss-Catalan paths (\cite{Liaw98,An14,Park16,Armstrong09}). An, Jung, and Kim (\cite{An17}) enumerated the number of small \kkfs paths of type $\lambda$ and posed the question of enumerating large \kkfs paths of type $\lambda$, which is achieved in the present paper through the use of generating functions. Furthermore, we enumerate the number of large \kkfs paths of type $\lambda$ ending on a diagonal step and generalize Fuss-\Sch paths to \kkfs paths for any subset $S\subseteq[k]$. The problem of enumeration by type here falls to the same techniques.

For future work, we would be interested in a combinatorial or bijective proof of the new formulae presented here, including the number of large \kkfs paths of type $\lambda$, the number of large \kkfs paths of type $\lambda$ ending on a diagonal step, and the number of $(k,S)$-Fuss-\Sch paths of type $\lambda$ of all three varieties (large, small, and ending on a diagonal step).

\section{Acknowledgments}
This work was supported by NSF grant DMS-1659047 as part of the Duluth Research Experience for Undergraduates (REU). The author is grateful to Benjamin Gunby for his help in the revision process as well as Joe Gallian for his helpful feedback and for administering the research program.

\bibliographystyle{plain}
\bibliography{anc/fussschroder}

\begin{thebibliography}{10}

\bibitem{An14}
Su~Hyung An, Sen-Peng Eu, and Sangwook Kim.
\newblock Large {S}chr{\"o}der paths by types and symmetric functions.
\newblock 51(4):1229--1240, 2014.

\bibitem{An17}
Suhyung An, JiYoon Jung, and Sangwook Kim.
\newblock Enumeration of {F}uss-{S}chr{\"o}der paths.
\newblock {\em The Electronic Journal of Combinatorics}, 24(2):\#P2--30, 2017.

\bibitem{Armstrong09}
Drew Armstrong.
\newblock Generalized noncrossing partitions and combinatorics of {C}oxeter
  groups.
\newblock {\em Memoirs of the American Mathematical Society}, 202(949), 2009.

\bibitem{Eu08}
Sen-Peng Eu and Tung-Shan Fu.
\newblock Lattice paths and generalized cluster complexes.
\newblock {\em Journal of Combinatorial Theory, Series A}, 115(7):1183--1210,
  2008.

\bibitem{Fomin05}
Sergey Fomin and Nathan Reading.
\newblock Generalized cluster complexes and {C}oxeter combinatorics.
\newblock {\em International Mathematics Research Notices},
  2005(44):2709--2757, 2005.

\bibitem{Fomin03}
Sergey Fomin and Andrei Zelevinsky.
\newblock Y-systems and generalized associahedra.
\newblock {\em Annals of Mathematics}, 158(3):977--1018, 2003.

\bibitem{Kreweras72}
Germain Kreweras.
\newblock On uncrossed partitions of a cycle.
\newblock {\em Discrete Mathematics}, 1(4):333--350, 1972.

\bibitem{Liaw98}
S.C. Liaw, H.G. Yeh, F.K. Hwang, and G.J. Chang.
\newblock A simple and direct derivation for the number of noncrossing
  partitions.
\newblock {\em Proceedings of the American Mathematical Society},
  126(6):1579--1581, 1998.

\bibitem{Park16}
Youngja Park and Sangwook Kim.
\newblock Chung-{F}eller property of {S}chr{\"o}der objects.
\newblock {\em The Electronic Journal of Combinatorics}, 23(2):2--34, 2016.

\bibitem{Stanley99}
Richard Stanley.
\newblock {\em Enumerative Combinatorics, Volume 2}.
\newblock Cambridge University Press, 1999.

\end{thebibliography}

\end{document}